\documentclass[a4paper,11pt]{article}

\usepackage{soul}
\usepackage{geometry}
\usepackage[T1]{fontenc}
\usepackage[english]{babel}
\usepackage{amssymb,amsmath,amsfonts,amstext,amsthm} 
\usepackage{bbm} 
\usepackage{color,hyperref} 
\usepackage[numbers,sort&compress]{natbib} 
\usepackage{authblk} 
\usepackage{cleveref} 
\usepackage{todonotes} 
\usepackage{array} 
\usepackage{booktabs} 
\usepackage{adjustbox} 
\usepackage{multirow} 

\definecolor{darkblue}{rgb}{0.0,0.0,0.3} 
\hypersetup{colorlinks,breaklinks,linkcolor=darkblue,urlcolor=darkblue,anchorcolor=darkblue,citecolor=darkblue} 
\newcommand{\panel}[1]{({\textbf{#1}})} 
\newcommand{\abs}[1]{\left|#1\right|} 
\newcommand{\E}{\mathbbm{E}} 
\newcommand{\R}{\mathbbm{R}} 
\newcommand{\K}{\tilde\kappa} 
\newcommand{\prob}{\operatorname{Pr}} 
\newcommand{\var}{\operatorname{var}} 
\newcommand{\ld}[1]{\log(\det {#1})} 
\newcommand{\Y}[2]{Y_{#2}(#1)} 
\newcommand{\p}{\mu} 
\newcommand{\cv}[1]{c_v(#1)} 

\crefname{section}{Section}{Sections} 
\crefname{cor}{Corollary}{Corollaries}
\crefname{equation}{eq.\hspace{-0.05cm}}{eqs.\hspace{-0.05cm}} 
\newcommand{\eq}[1]{\Cref{#1}} 

\newtheorem{theorem}{Theorem}
\newtheorem{prop}{Proposition}
\newtheorem{cor}{Corollary}

\newtheorem{conj}{Conjecture}
\theoremstyle{remark}
\newtheorem{remark}{Remark}

\makeatletter
\renewenvironment{proof}[1][\proofname]{%
   \par\pushQED{\qed}\normalfont%
   \topsep6\p@\@plus6\p@\relax
   \trivlist\item[\hskip\labelsep\bfseries#1\@addpunct{.}]%
   \ignorespaces
}{%
   \popQED\endtrivlist\@endpefalse
}
\makeatother

\begin{document}
\title{Log-minor distributions and an application to estimating mean subsystem entropy}
\author[1]{Alice C. Schwarze}
\author[2]{Philip S. Chodrow}
\author[3]{Mason A. Porter}
\affil[1]{Wolfson Centre for Mathematical Biology, Mathematical Institute, University of Oxford, Oxford OX2 6GG, UK}
\affil[2]{Operations Research Center, Massachusetts Institute of Technology, Cambridge, MA, 02139}
\affil[3]{Department of Mathematics, University of California, Los Angeles, 520 Portola Plaza, Los Angeles, California 90095, USA}
\maketitle



\section*{Abstract}

A common task in physics, information theory, and other fields is the analysis of properties of subsystems of a given system. Given the covariance matrix $M$ of a system of $n$ coupled variables, the covariance matrices of the subsystems are principal submatrices of $M$. The rapid growth with $n$ of the set of principal submatrices makes it impractical to exhaustively study each submatrix for even modestly-sized systems. It is therefore of great interest to derive methods for approximating the distributions of important submatrix properties for a given matrix. 

Motivated by the importance of differential entropy as a systemic measure of disorder, we study the distribution of log-determinants of principal $k\times k$ submatrices when the covariance matrix has bounded condition number. We derive upper bounds for the right tail and the variance of the distribution of minors, and we use these in turn to derive upper bounds on the standard error of the sample mean of subsystem entropy. Our results demonstrate that, despite the rapid growth of the set of subsystems with $n$, the number of samples that are needed to bound the sampling error is asymptotically independent of $n$. Instead, it is sufficient to increase the number of samples in linear proportion to $k$ to achieve a desired sampling accuracy. 


\section{Introduction}\label{sec:introduction}

In many fields of study, researchers use matrices to represent systems of interest. Statisticians and data scientists represent large tabular data sets as matrices \cite{Skiena2017}. In network science, it is common to use adjacency matrices to represent the structure of a network \cite{Newman2018}. In dynamical systems, researchers use Jacobian matrices in the study of the linearized dynamics of a system of coupled variables \cite{Strogatz2015}. For networks, dynamical systems, statistical analysis of large data sets, and other applications, it can be insightful (and even necessary) to examine their components (as subnetworks, subsystems, reduced data sets, and so on). Several researchers have used subsystem properties to characterize robustness and other salient properties of dynamical systems \cite{Tononi1994,Tononi1999,Lucia2005,Randles2011,Li2012,Li2016}. Network scientists count and analyze motifs and other subgraphs in networks to characterize a network's structure \cite{Shen-Orr2002,Milo2002}. Several prominent tools in data science are based on linear sketching, an approach to data dimensionality reduction whereby one obtains a reduced data set via matrix multiplication \cite{Woodruff2014,Francis2018} or as a linear combination of submatrices \cite{Sarlos2006,Clarkson2009,Kyrillidis2014}. An example of such a tool for dimensionality reduction is principal component analysis \cite{Liberty2013}. 

The various applications of submatrices motivate the mathematical study of their properties. In this paper, we study the distribution of log-determinants of principal submatrices of a positive definite matrix and show that our results lead to controllable sampling guarantees for computing the mean differential entropy of subsystems for a dynamical system. Researchers have studied the differential entropy of subsystems in areas such as physics \cite{Page1993,Sen1996}, biology \cite{Li2012,Li2016}, neuroscience \cite{Tononi1994,Tononi1999,Lucia2005}, computer science \cite{Randles2011}, and coding theory \cite{Koetter2003}.  For example, Tononi et al.\,(1999) computed a measure of network redundancy from the mean differential entropy of its subsystems of fixed size \cite{Tononi1999}. Teschendorff et al.\,(2014) \cite{Teschendorff2014} used differential entropy to define a measure of network robustness for protein-interaction networks .

For several symmetric multivariate distributions, estimates of differential entropy are affine functions of the log-determinant of a system's covariance matrix. Examples include the multivariate normal distribution \cite{Ahmed1989}, the multivariate $t$ distribution \cite{guerrero1996measure,Nadarajah2005}, and the multivariate Cauchy distribution \cite{Nadarajah2005}. For the $n$-variate normal distribution with covariance matrix $M$, for example, the differential entropy is \cite{Ahmed1989}
\begin{align}
    h(M) = \frac{1}{2} \log(\det M) + \frac{n}{2}\left(1 + \log (2\pi)\right)\,,  \label{eq:diff_ent}
\end{align}
where the base of the logarithm can be any finite positive number\footnote{The base $b$ of the logarithm determines the units of entropy. If one chooses $b=2$, one measures entropy values in bits. If one chooses $b=e$, one measures entropy values in nats.}.
The logarithm of the covariance matrix is thus sufficient to approximate the differential entropy of several multivariate distributions. 

The principal submatrices of $M$ are covariance matrices of subsystems that correspond to subsets of coupled variables. One can compute the differential entropy of a subsystem by computing $h$ in \Cref{eq:diff_ent} for a principal submatrix of $M$. A system of $n$ coupled variables possesses $\binom{n}{k}\approx \frac{n^k}{k!}$ subsystems of $k$ variables; each of these subsystems corresponds to one of the $\binom{n}{k}$ principal $k\times k$ submatrices of $M$. The exact computation of the distribution of differential subsystem entropy or its moments thus requires one to compute $O(n^k)$ distinct determinants, an infeasible task for large $n$ and $k$. This task can be computationally prohibitive even for modestly-sized systems. To our knowledge, the largest system for which researchers have exactly computed the differential entropy of subsystems is a synthetic network with $n=12$ variables and subsystems with $k\leq 12$ variables \cite{Tononi1999}. 

To address this problem, we study the distribution of log-determinants of principal $k\times k$ submatrices. We refer to these log-determinants as \textit{log-minors} of \textit{size} $k$. As we noted above, these log-determinants are sufficient to determine the subsystem entropy for many important multivariate distributions. Knowledge of the properties of this distribution thus enables the derivation of bounds on the sampling error when estimating subsystem entropy in many applications. We show that, given a bound on the condition number of $M$, the standard error of a sample mean of differential entropy is independent of $n$ and sublinear in $k$, implying that one needs a sublinear number of samples in $k$ to ensure a desired accuracy.

Our paper proceeds as follows. In \Cref{sec:notation}, we introduce some notation that we use throughout this paper. In \Cref{sec:results}, we give several upper bounds on the tail and variance of the distribution of log-minors of a positive-definite matrix with bounded condition number. We present proofs for these bounds in \Cref{sec:proofs} and show numerical examples in \Cref{sec:examples}. In \Cref{sec:sampling}, we apply our theorems to provide probabilistic guarantees on the sample mean and relative error, and we discuss implications for the design of practical schemes for estimating mean subsystem entropy. We conclude and discuss possible extensions in \Cref{sec:conclusions}. 


\section{Notation}\label{sec:notation}

Let $M \in \R^{n\times n}$ be a positive-definite matrix. Let $\lambda_1(M)\geq\lambda_2(M)\cdots\geq\lambda_n(M)\geq 0$ be the eigenvalues of $M$. Because $M$ is positive definite, it is also nonsingular; its condition number is $\kappa(M) = \lambda_1(M)/\lambda_n(M)$. For a given index set $I\in[n]^k := \{1, \ldots, n\}^k$, the matrix $M_{I} := [M_{i,j}]_{i,j\in I}$ is the corresponding principal submatrix of $M$. For any fixed $k\leq n$, let $\mathcal{A}_k(M)$ denote the set of all such $k\times k$ submatrices of $M$, and let $A_k(M)$ denote a uniformly-random element of this set. We define a random variable $\Y{M}{k} := \log(\det A_k(M))$ and denote its empirical distribution by $\p_{M,k}$. For convenience, we define $\wedge_{n,k}:=\min\{k,n-k\}$.


\section{Bounds on the distribution of log-minors}\label{sec:results}

In this section, we state bounds on the distribution $\p_{M,k}$ for a positive-definite matrix $M$ with bounded condition number. We give upper bounds for the distribution's support, variance, and right tail. We also show that we can improve these bounds if $M$ is diagonal.

\begin{theorem}[Tail and variance bound for log-minors of a positive-definite matrix]
Let $M$ be a positive-definite $n\times n$ matrix with condition number $\kappa(M)\leq\K$. For every $r\geq 0$, we have
\begin{equation}
    \prob\left(|\Y{M}{k} -\mathbbm E [\Y{M}{k}]|\geq r\right)\leq 3\exp \left(-\frac{r}{\log\K}\sqrt{\frac{n}{k(n-k)}}\right)\,.\label{eq:ldi}
\end{equation}
Furthermore, the variance of $\Y{M}{k}$ satisfies  
\begin{equation}
    \var(\Y{M}{k})\leq 6\left(\frac{k(n-k)}{n}\right)(\log\K)^2\,.
    \label{eq:ldi_varbound}
\end{equation}
\label{th:ldi}\end{theorem}

\begin{remark} 
The tail bound in \Cref{eq:ldi} does not guarantee that $\Y{M}{k}$ concentrates\footnote{Ledoux defined concentration of measure in Ref.\,\cite{Ledoux2001} (on page 3) as follows. Let $(X,d)$ be a metric space with probability measure $\mu$ on Borel sets of $(X,d)$. The \textit{concentration function} is defined as $\alpha_{(X,d,\mu)}:=\sup\{1-\mu(A_r);A\subset X, \mu(A)\geq\frac{1}{2}\}$, where $r>0$ and $A_r:=\{x\in X;d(x,A)<r\}$ is the open $r$-neighborhood of $A$. The measure $\mu$ has \textit{normal concentration} on $(X,d)$ if there are constants $c$ and $C$ such that $\alpha_{(X,d,\mu)}\leq C\exp(-cr^2)$ for every $r$.} on $\mathcal A_k(M)$. This is because the bound in \Cref{eq:ldi} is increasing with respect to $k$ and asymptotically constant with respect to $n$. Indeed, for the bound to approach $0$ for a sequence $\{M_i\}$ of matrices, it is both necessary and sufficient that $\lim_{i\rightarrow \infty} \sqrt{k_i}\log \kappa_i(M) \rightarrow 0$. Because $k$ cannot be smaller than $1$, this condition requires the condition number $\kappa_i(M)$ to approach $1$. The condition $\lim_{i\rightarrow\infty}\kappa_i(M)\rightarrow 1$ severely constrains the sequence $\{M_i\}$. In that limit, all eigenvalues of $M$ are equal to each other and all log-minors are equal to $k\log\lambda_1$.
\end{remark}

\begin{theorem}[Support and variance bound for log-minors of a positive-definite matrix] 
Let $M$ be a positive-definite $n\times n$ matrix with condition number $\kappa(M)\leq\K$. For any $k$, the random variable $\Y{M}{k}$ and its distribution $\p_{M,k}$ satisfy the following properties:
\begin{enumerate}
    \item The distribution $\p_{M,k}$ has bounded support that is contained in an interval whose length is no greater than $(\wedge_{n,k}\times\log\K)$; and
    \item the variance of $\Y{M}{k}$ satisfies 
    \begin{equation}
        \var(\Y{M}{k})\leq \frac{1}{4}\left(\wedge_{n,k}\times \log \K\right)^2\,.
        \label{eq:support_varbound}
    \end{equation}
\end{enumerate}
\label{th:support}\end{theorem}

\begin{remark} 
The variance bound in \Cref{eq:support_varbound} is much sharper than the one in \Cref{eq:ldi_varbound}. Both variance bounds are asymptotically constant with respect to $n$. For fixed $k$, the two variance bounds differ by a factor of $24$ in the large-$n$ limit.
\end{remark}

\begin{remark} 
For even $n$ and $k\in\{1,n-1\}$, the bound on the variance in \eq{eq:support_varbound} is sharp when $M$ is a $2\ell\times 2\ell$ (where $\ell=n/2$) diagonal matrix with entries $\lambda_1 = \cdots = \lambda_\ell = \K$ and $\lambda_{\ell+1} = \cdots = \lambda_{2\ell}=1$. 
\end{remark}

When $M$ is diagonal, we can derive a variance bound that is sharper than the bounds in \Cref{eq:ldi_varbound,eq:support_varbound}.

\begin{theorem}[Variance bound for log-minors of a positive-definite diagonal matrix]
Let $D$ be a positive-definite $n\times n$ diagonal matrix with condition number $\kappa(D)\leq\K$. The variance of $\Y{D}{k}$ satisfies  
\begin{equation}
    \var(\Y{D}{k})\leq\frac{k}{4}\left(\frac{n-k}{n-1}\right)(\log\K)^2\,.
    \label{eq:diag_varbound}
\end{equation}
\label{th:diag}\end{theorem}

\begin{remark} 
The two variance bounds in \Cref{eq:support_varbound,eq:diag_varbound} are asymptotically constant with respect to $n$ and converge to the same limiting value of $k(\log\K)^2/4$. 
\end{remark}

\begin{remark} 
The variance bound for diagonal positive-definite matrices in \Cref{eq:diag_varbound} is sharper than the variance bound for general positive-definite matrices in \Cref{eq:support_varbound}. The former differs from the latter by a factor of $\max\{k,n-k\}/(n-1)\leq 1$. 
\end{remark}

\begin{remark} 
For even $n$ and any $k\leq n$, the bound on the variance in \eq{eq:diag_varbound} is sharp when $M$ is a $2\ell\times 2\ell$ (where $\ell=n/2$) diagonal matrix with entries $\lambda_1 = \cdots = \lambda_\ell = \K$ and $\lambda_{\ell+1} = \cdots = \lambda_{2\ell}=1$. The sharpness of the bound for diagonal matrices indicates a limit to possible improvements for the variance bound for general positive-definite matrices. Specifically, one cannot hope to improve the variance bound in \Cref{eq:support_varbound} by more than a factor of $\max\{k,n-k\}/(n-1)$.
\end{remark}
   
The variance bound in \Cref{th:support} is sharp for a diagonal matrix. This observation and several examples in \Cref{sec:examples} motivate the following conjecture. 

\begin{conj}[Diagonal matrices maximize log-minor variance] Let $\mathcal M_{\kappa}$ be the set of positive-definite $n\times n$ matrices with condition number $\kappa$. For all $k<n$, $\kappa$, and $M\in\mathcal M_{\kappa}$, there exists a diagonal matrix $D\in\mathcal M_{\kappa}$ such that
\begin{equation}
    \var(\Y{M}{k})
        \leq 
    \var(\Y{D}{k})\,.
\label{eq:ineq}
\end{equation}
    \label{conj:diag}
\end{conj}

The variance bounds (see \,\Cref{eq:support_varbound,eq:ldi_varbound,eq:diag_varbound}) have important implications for the accuracy of sample means of log-minors. We discuss these implications in \Cref{sec:sampling}.


\section{Proofs of bounds on the distribution of log-minors}\label{sec:proofs}

\subsection{Proof of \Cref{th:ldi}}\label{sec:proof2}

To prove \Cref{th:ldi}, we use Cauchy's interlacing theorem and results on Markov chains on countable sets. Chatterjee and Ledoux (2009) previously used this approach to prove a concentration result for empirical cumulative eigenvalue spectra of Hermitian matrices \cite{Chatterjee2009}.

\begin{prop}[Cauchy's interlacing theorem \cite{Hwang2004}]\label{prop:cauchy} 
Let $M$ be a Hermitian $n\times n$ matrix, and let $A$ be a principal $(n-1)\times (n-1)$ submatrix of $M$. 
If $M$ has eigenvalues $\lambda_1(M)\geq\lambda_2(M)\geq\dots\geq\lambda_n(M)$ and $A$ has eigenvalues $\lambda_1(A)\geq\lambda_2(A)\geq\dots\geq\lambda_{n-1}(A)$, then
\begin{equation}
    \lambda_1(M)\geq\lambda_1(A)\geq\lambda_2(M)\geq\lambda_2(A)\geq\lambda_3(M)\geq\dots\geq\lambda_{n-1}(M)\geq\lambda_{n-1}(A)\geq\lambda_n(M)\,.\nonumber
\end{equation}
\end{prop}

\begin{prop}[Large-deviation inequality for functions on countable sets \cite{Ledoux2001} (page 50)]\label{prop:ledoux}
Let $(\Pi,\mu_{\mathcal A})$ be a reversible Markov chain on a finite or countable set $\mathcal A$. Let $(\Pi,\mu_{\mathcal A})$ have a spectral gap\footnote{The spectral gap (also called the ``Poincar\'e constant'') of a Markov chain $(\Pi,\mu_{\mathcal A})$ on a space $\mathcal A$ is the constant $\overline{\lambda}$ such that, for all functions $f$, we have $\overline{\lambda}\times\var_\mu(f)\leq\frac{1}{2}\sum_{A,B\in\mathcal A}[f(A)-f(B)]^2\Pi(A,B)\mu(\{A\})$\,. See, for example, Ref.\,\cite{Ledoux2001} (page 50).} of $\overline{\lambda}>0$. It follows that , whenever $f: \mathcal A\rightarrow\mathbbm R$ is a function such that 
\begin{equation}
    |||f|||^2_\infty:=\frac{1}{2}\sup_{A\in \mathcal A}\sum_{B\in \mathcal A} \left(f(A)-f(B)\right)^2\Pi(A,B)\leq 1\,,\nonumber
\end{equation} 
 it is also true that $f$ is integrable with respect to $\mu_{\mathcal A}$ and that, for every $r\geq 0$, the probability measure 
\begin{equation}
    \mu\left(\{f\geq \int_{\mathcal A} f\,d\mu_{\mathcal A} + r\}\right)\leq 3e^{-\frac{r}{2}\sqrt{\overline{\lambda}}}\,. \nonumber
\end{equation}
\end{prop}

\begin{remark} 
The expected squared distance in $f$ between $A\in\mathcal A$ and its adjacent states in the Markov chain is $\sum_{B\in\mathcal A}\left(f(A)-f(B)\right)^2\Pi(A,B)$. One can thus think of $|||f|||_{\infty}^2$ as a measure of the expected squared distance between the greatest ``outlier'' $A$ and adjacent states in $\mathcal A$. We thus refer to $|||f|||_{\infty}^2$ as the \textit{squared outlier deviation} of $f$ on $(\Pi,\mu_{\mathcal A})$.
\end{remark}

\begin{prop}[Spectral gap of random-transposition walk \cite{Diaconis1981}]\label{prop:gap}
Let $\mathcal S_n$ be the set of permutations of $n$ elements, and let $s\in\mathcal S_n$. Let the ``random-transposition walk'' be a reversible Markov chain $(\Pi, \mu_{\mathcal S_n})$ with kernel
\begin{align}\Pi=\left\{\begin{matrix}
    1/n\,,&\textrm{if } s'=s\,,\\
2/n^2\,,&\textrm{if } s'=\tau s \textrm{ for some transposition }\tau\,,\\
0\,,&\textrm{otherwise}\,.
    \end{matrix}\right. \nonumber
\end{align}
The random-transposition walk has a spectral gap of $\overline \lambda=2/n$.
\end{prop}

\begin{proof}[Proof of \Cref{th:ldi}]
Every principal $k\times k$ submatrix of $M$ is the top-left principal $k\times k$ submatrix of $M$ after a permutation of its rows and columns. We denote the permutated matrix by $s M$ and its top-left principal $k\times k$ submatrix by $\hat A_k(s M)$, where $s\in \mathcal S_n$ is a permutation of $n$ elements.

For the top-left principal $k\times k$ submatrix, only the first $k$ elements of $s$ are relevant. There are $(n-k)!$ permutations $s\in \mathcal S_n$ that are identical in their first $k$ elements, so there is a $1$-to-$(n-k)!$ correspondence between $\mathcal A_k(M)$ and $\mathcal S_n$. Because of the correspondence between $\mathcal A_k(M)$ and $\mathcal S_n$, we obtain the same distribution for a function $f(A_k(M))$, where we choose $A_k(M)$ uniformly at random from $\mathcal A_k(M)$, and for $f(\hat A_k(sM))$, where we choose $s$ uniformly at random from $\mathcal S_n$. 

Let $f_{\alpha}: \mathcal S_n \rightarrow \mathbbm R$ be such that
\begin{equation} \label{eq:fdef}
    f_{\alpha}(s) := \alpha\ld{\hat A_k(sM)}
\end{equation}
for some $\alpha\in\mathbbm R$. To find an upper bound on the squared outlier deviation for $f_{\alpha}$ on the random-transposition walk, we make two observations:
\begin{enumerate}
\item Consider two permutations, $s$ and $s'$, that are adjacent in the random-transposition walk; that is, $s'=\tau s$ for some transposition $\tau$. The determinant is invariant under basis transformation, so the value of $f_{\alpha}(s')$ can differ from $f_{\alpha}(s)$ only if $\tau$ is a transposition that swaps one of the first $k$ elements in $s$ with one of the last $n-k$ elements in $s$. There are $n^2$ possible transpositions for a sequence of $n$ elements; and $2k(n-k)$ of these transpositions swap one of the first $k$ elements of the sequence with one of the last $n-k$ elements of the sequence. Consequently, the fraction of transpositions that change the value of $f_{\alpha}$ has an upper bound of 
\begin{equation}
    b_1:= \frac{2k}{n^2}(n-k)\,.\nonumber
\end{equation}
\item Using Cauchy's interlacing theorem (see \Cref{prop:cauchy}), one can find an upper bound $b_2$ for $|f_{\alpha}(A)-f_{\alpha}(B)|$. For any $k<n$ and any pair $A, B\in\mathcal A_k(M)$, there exists a matrix $C\in A_{k+1}(M)$ such that $A$ and $B$ are principal submatrices of $C$. Cauchy's interlacing theorem implies that
\begin{enumerate}
    \item $\lambda_1(M)$ is an upper bound on the largest eigenvalue of $C$;
    \item $\lambda_n(M)$ is a lower bound on the smallest eigenvalue of $C$;
    \item $\sum_{i=1}^{k}\log\lambda_i(C)$ is an upper bound on $\ld{A}$ and $\ld B$; and
    \item $\sum_{i=2}^{k+1}\log\lambda_i(C)$ is a lower bound on $\ld{A}$ and $\ld B$.
\end{enumerate}
Therefore,
\begin{align*}
    |\ld{A}-\ld B|\leq & \sum_{i=1}^{k}\log\lambda_i(C) - \sum_{i=2}^{k+1}\log\lambda_i(C)\\
    \leq & \log\lambda_1(M) - \log\lambda_n(M) \leq \log \K\,.
\end{align*}
This upper bound for $|\ld{A}-\ld B|$ holds for arbitrary $A,B\in\mathcal A_k(M)$. We can thus set the upper bound to be $b_2=\alpha\log\K$.
\end{enumerate}
We obtain an upper bound for the squared outlier deviation of $f_{\alpha}$ of
\begin{align*}
    |||f_{\alpha}|||_\infty^2 &\leq \frac{1}{2}b_1b_2^2\\
&=k(n-k)\left(\frac{\alpha\log\K}{n}\right)^2\,.
\end{align*}

Let $\alpha'=n\log\K/\sqrt{k(n-k)}$. The function $f_{\alpha'}$ on $(\Pi,\mu_{\mathcal S_n})$ has a squared outlier deviation of $|||f_{\alpha'}|||^2_\infty\leq 1$. We can thus use the tail bound for functions on countable sets (see \Cref{prop:ledoux}) for $f_{\alpha'}$. Therefore,
\begin{equation*}
    \prob(|f_{\alpha'}(\hat A_k(\sigma M)) -\mathbbm E[f_{\alpha'}(A_k(\sigma M)]|\geq \alpha'r)\leq 3e^{-\frac{\alpha'r}{2}\sqrt{\overline{\lambda}}}\,.
\end{equation*}
We can substitute $f_{\alpha'}(\hat A_k(\sigma M))$ in \eq{eq:fdef} by $\alpha'\Y{M}{k}$, because of the correspondence between $\mathcal A_k(M)$ and $\mathcal S_n$. Applying \Cref{prop:gap}, we obtain
\begin{align*}
    \prob(|\alpha'\Y{M}{k} -\alpha'\mathbbm E [\Y{M}{k}]\geq \alpha'r)&\leq 3e^{-\frac{\alpha'r}{\sqrt{2n}}}\\
    \implies\prob(|\Y{M}{k} -\mathbbm E[\Y{M}{k}]|\geq r)&\leq 3\exp \left(-\frac{r}{\log\K}\sqrt{\frac{n}{k(n-k)}}\right)\,.
\end{align*}
This proves the first statement of \Cref{th:ldi} (see \eq{eq:ldi}).

We derive a bound on the variance of $\ld{A}$ from \Cref{eq:ldi} from a direct calculation. First, we write 
\begin{align*}
    \var(\Y{M}{k})  &= \mathbbm{E}[(\Y{M}{k}-\mathbbm{E}[\Y{M}{k}])^2]\\
    &= \int_{0}^\infty \prob(\Y{M}{k}-\mathbbm{E}[\Y{M}{k}])^2 \geq u)du\\
    &= \int_{0}^\infty \prob(\Y{M}{k}-\mathbbm{E}[\Y{M}{k}])\geq \sqrt u\} du\,.
\end{align*}
Using the tail bound in \eq{eq:ldi}, it follows that
\begin{align*}
    \var(\Y{M}{k})  &\leq \int_{0}^\infty3\exp \left(-\frac{\sqrt{u}}{\log\K}\sqrt{\frac{n}{k(n-k)}}\right) du\\
    & = 6\left(\frac{k(n-k)}{n}\right)(\log\K)^2
\end{align*}
\end{proof}

\subsection{Proof of \Cref{th:support}}\label{sec:proof1}

We prove \Cref{th:support} using Cauchy's interlacing theorem and Popoviciu's inequality.
\begin{prop}[Popoviciu's inequality \cite{Popoviciu1935, Sharma2010}]\label{prop:popoviciu}
Let $X$ be a real-valued random variable supported on the interval $[x_\mathrm{min}, x_{\mathrm{max}}]$. It then follows that $X$ has variance
\begin{equation}
    \var(X)\leq (x_{\operatorname{max}}-x_{\operatorname{min}})^2/4\,.\nonumber
\end{equation}
\end{prop}
For a proof of this version of Popoviciu's inequality, see Ref.\,\cite{stackPopoviciu}.

\begin{proof}[Proof of \Cref{th:support}] 
For any finite $n$ and $k$, the set $\mathcal A_k(M)$ of principal $k\times k$ submatrices of an $n\times n$ matrix $M$ has finite cardinality $\binom{n}{k}$. It follows that the distribution of any function of $A_k(M)$ has finite support. We define an interval $[r_1,r_2]$ with $r_1:=\min\Y{M}{k}$ and $r_2:=\max\Y{M}{k}$, such that the support of $\p_{M,k}$ is a finite subset of $[r_1,r_2]$.

We can obtain any principal $k\times k$ submatrix $A$ of $M$ by removing $n-k$ row--column pairs from $M$. Successive applications of Cauchy's interlacing theorem show that $\lambda_1(A)\leq\lambda_1(M)$ and $\lambda_k(A)\geq\lambda_n(M)$. It follows that
\begin{align}
    r_1=k\log\lambda_n(M)\,, \qquad
    r_2=k\log\lambda_1(M)\,.\nonumber
\end{align}
Therefore,
\begin{align}
    r_2-r_1\leq k\left(\log\K+\log\lambda_n(M)\right)-k\log\lambda_n(M)=k\log \K\,.\nonumber
\end{align}
If $k>n/2$, any two principal $k\times k$ submatrices share $2k-n$ rows and columns. They can thus differ in at most $n-k$ rows and columns. It follows that one can refine the lower and upper bounds on the support of $\p_{M,k}$ so that $[r_1,r_2]\leq\wedge_{n,k}\times\log\K$. We have thus proven the first part of Theorem 1. Applying Popiviciu's inequality to $X=\Y{M}{k}$ with $x_{\operatorname{max}}-x_{\operatorname{min}}\leq r_2-r_1$ yields the variance bound in \Cref{th:support}.
\end{proof}
 
\subsection{Proof of \Cref{th:diag}}\label{sec:proof3}
For our proof of \Cref{th:diag}, we maximize $\var(\Y{D}{k})$ with respect to the eigenvalues of $D$.

\begin{proof}[Proof of \Cref{th:diag}]
Let $D$ be a positive-definite diagonal matrix with entries $\lambda_1(D)\geq\lambda_2(D)\geq\dots\geq\lambda_n(D)>0$. 
Define $x_i = \log \lambda_i(D)$ for each $i \in [n]$; and let $I \in [n]^k$. It then follows that
\begin{align}
    \ld{D_{I}} &= \sum_{i \in I}x_{i}\,.\label{eq:xi}
\end{align}
We now consider the function $v(x_1,x_2,\ldots,x_n) := \var(\Y{D}{k})$. From \Cref{eq:xi}, we see that every value of $\Y{D}{k}$ is a sum of a subset of the variables $x_i$. Therefore, the function $v(x_1,x_2,\ldots,x_n)$ is convex (i.e., concave up) in the variables $x_1,x_2,\ldots,x_n$. Furthermore, the variance is translation-invariant. We may therefore, without loss of generality, suppose that $x_n = 0$ (corresponding to $\lambda_n(D) = 1$) and $x_1 = \log\kappa(D)$ (corresponding to $\lambda_n(D) =\kappa(D)$). Consequently, the maximization of the variance $\var(\Y{D}{k})$ amounts to the maximization of $v$ over the volume $[0,\log \kappa(D)]^n$ associated with an $n$-dimensional hypercube with edge length $\log \kappa(D)$. The solutions lie at the vertices of this hypercube. Therefore,
\begin{align*}
    x_i = \begin{cases}
        \log \kappa(D)\,, &\quad i \leq \ell \,, \\ 
        0 \,, &\quad \text{otherwise}\,,
    \end{cases}
\end{align*}
for some $\ell \in [n]$. We may now view $\Y{D}{k}/(\log \kappa(D))$ as a hypergeometric random variable on a population of size $n$ for which $\ell$ elements have the value $1$ and $n-l$ elements have the value $0$. The variance of this hypergeometric random variable is
\begin{align*}
    \var(\Y{D}{k}) = \frac{k(n-k)}{n^2(n-1)}\left(n-\ell\right)\ell (\log \kappa(D))^2\,,
\end{align*}
which is maximal at
\begin{align*}
    \ell^* = 
    \begin{cases}
        \frac{n}{2} \,, &\quad n \text{ even}\,, \\ 
        \frac{n\pm1}{2}\,, &\quad n\text{ odd}\,.
    \end{cases}
\end{align*}
The maximal value of $\var(\Y{D}{k})$ for even $n$ leads to the variance bound
\begin{align} \label{eq:vbd}
    \var (\Y{D}{k}) &\leq \frac{k(n-k)}{4(n-1)}(\log \kappa(D))^2 \\
    &\leq\frac{k(n-k)}{4(n-1)}(\log \K)^2\,. 
\end{align}
Comparing the maximal values of $\var(\Y{D}{k})$ for even $n$ and for odd $n$ shows that \Cref{eq:vbd} is a variance bound for all $n$.
\end{proof}


\section{Examples}\label{sec:examples}

In this section, we compare the tail of the distribution $\p_{M,k}$ for several example matrices $M$ to the bounds in Theorems \ref{th:ldi} and \ref{th:diag}.

We consider four examples of positive-definite $n\times n$ matrices with $n=20$ and fixed condition number $\kappa=3$. \paragraph{Example E1.} Consider the diagonal matrix $M_{\textrm{E1}}$ that maximizes the variance of $\Y{M_{\textrm{E1}}}{k}$.  (See the proof of \Cref{th:diag}.) For even $n$, this matrix has eigenvalues $\lambda_1, \lambda_2, \ldots, \lambda_n$, where $\lambda_1,\ldots,\lambda_{n/2}=\K$ and $\lambda_{n/2+1},\ldots,\lambda_n=1$. 

\paragraph{Example E2.} Consider a diagonal matrix $M_{\textrm{E2}}$ with eigenvalues $\lambda_1, \lambda_2, \ldots, \lambda_n$. We set $\lambda_1=\kappa$ and $\lambda_n=1$. We draw $\lambda_2,\ldots,\lambda_{n-1}$ from a uniform distribution on $[1,\kappa]$.

\paragraph{Example E3.} We obtain a non-diagonal positive-definite matrix $M_{\textrm{E3}}$ with condition number $\kappa$ via an orthogonal transformation of $M_{\textrm{E1}}$. That is,
\begin{equation}
    M_{\textrm{E3}} := Q^{-1}M_{\textrm{E1}}Q\,,\nonumber
\end{equation}
where $Q$ is an orthogonal matrix that we choose from the Haar measure over the group of orthogonal matrices. We use Stewart's algorithm \cite{Stewart1980} to generate $Q$.

\paragraph{Example E4.} We again generate a random orthogonal matrix $Q$ using Stewart's algorithm. We obtain another non-diagonal positive-definite matrix $M_{\textrm{E4}}:=Q^{-1}M_{\textrm{E2}}Q$ via an orthogonal transformation of $M_{\textrm{E2}}$.\vspace{1em}

\begin{figure}[p!tb]
\centering
\includegraphics[trim={0cm 1.5cm 0cm 2.75cm},clip,width=1\textwidth]{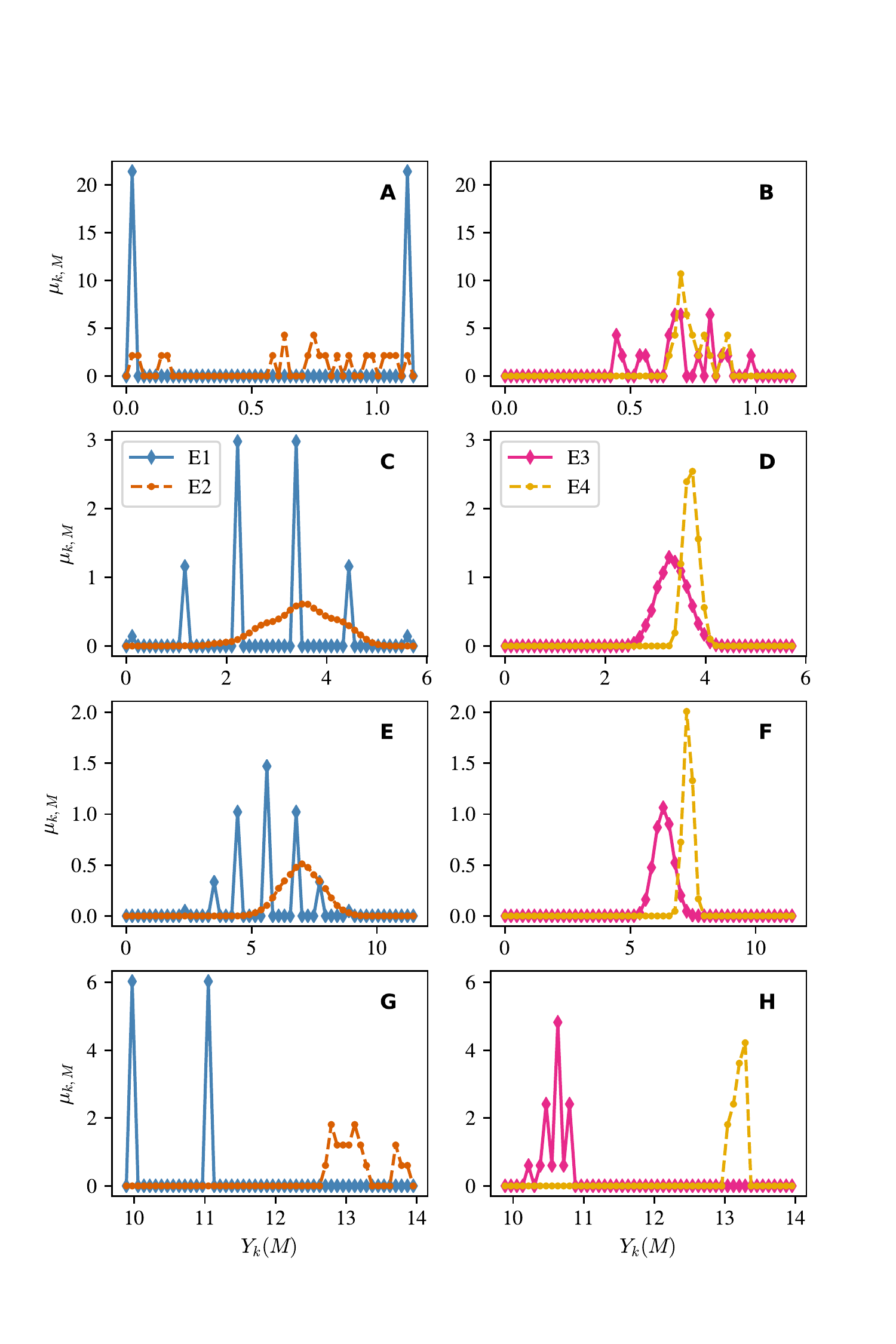}
\caption{Approximate probability density functions for log-minors. We show $\p_{M,k}$ for Examples E1, E2, E3, and E4 with $n=20$ and condition number $\kappa=3$. We show results for $k=1$ in panels \panel{A} and \panel{B}, for $k=5$ in panels \panel{C} and \panel{D}, for $k=10$ in panels \panel{E} and \panel{F}, and for $k=19$ in panels \panel{G} and \panel{H}.}
\label{fig:pdfs}
\end{figure}

In \Cref{fig:pdfs}, we show the empirical probability densities of $\Y{M}{k}$ for Examples E1, E2, E3, and E4 using four different values of $k$. For all four examples, we observe that the interval on which $\p_{M,k}$ is supported shifts to the right for progressively larger $k$. The length of the supported interval increases with $\wedge_{n,k}$. For $k=5$ and $k=10$ --- the cases in which $\wedge_{n,k}$ is larger than $1$ --- the distribution $\p_{M,k}$ are almost symmetric about $\E[\Y{M}{k}]$ for all four examples. For Example E1, the distribution $\p_{M,k}$ is symmetric about its mean for all examined values of $k$. Its density is nonzero at $\wedge_{n,k}+1$ equidistant points. 

\begin{table}[h!]
\begin{center}
\begin{tabular}{ccccccc}
    \hline\rule{0pt}{2.6ex}\rule[-1.2ex]{0pt}{0pt} 
    \multirow{2}{*}{$k$} 
    & \multirow{2}{*}{Example} 
    & \multirow{2}{*}{$\mathbbm E[\Y{M}{k}]$}
    & \multirow{2}{*}{$\var(\Y{M}{k})$} 
    & \multicolumn{3}{c}{Variance bound from} \\
    & & & 
    & Theorem 1 & Theorem 2 & Theorem 3 \\
    \hline\rule{0pt}{2.6ex}\rule[-1.2ex]{0pt}{0pt} 
    1 & E1 & $0.549$ & $0.302$ & $6.880$ & $0.302$ & $0.302$ \\\rule{0pt}{2.6ex}\rule[-1.2ex]{0pt}{0pt}
     & E2 & $0.689$ & $0.115$ & $6.880$ & $0.302$ & $0.302$ \\\rule{0pt}{2.6ex}\rule[-1.2ex]{0pt}{0pt}
     & E3 & $0.683$ & $0.021$ & $6.880$ & $0.302$ & N/A \\\rule{0pt}{2.6ex}\rule[-1.2ex]{0pt}{0pt}
     & E4 & $0.739$ & $0.005$ & $6.880$ & $0.302$ & N/A\\\rule{0pt}{2.6ex}\rule[-1.2ex]{0pt}{0pt}
    5 & E1 & $2.747$ & $1.191$ & $27.156$ & $1.509$ & $1.191$ \\\rule{0pt}{2.6ex}\rule[-1.2ex]{0pt}{0pt}
     & E2 & $3.446$ & $0.454$ & $27.156$ & $1.509$ & $1.191$ \\\rule{0pt}{2.6ex}\rule[-1.2ex]{0pt}{0pt}
     & E3 & $3.283$ & $0.091$ & $27.156$ & $1.509$ & N/A \\\rule{0pt}{2.6ex}\rule[-1.2ex]{0pt}{0pt}
     & E4 & $3.649$ & $0.020$ & $27.156$ & $1.509$ & N/A \\\rule{0pt}{2.6ex}\rule[-1.2ex]{0pt}{0pt}
    10 & E1 & $5.493$ & $1.588$ & $36.208$ & $3.017$ & $1.588$ \\\rule{0pt}{2.6ex}\rule[-1.2ex]{0pt}{0pt}
     & E2 & $6.893$ & $0.605$ & $36.208$ & $3.017$ & $1.588$ \\\rule{0pt}{2.6ex}\rule[-1.2ex]{0pt}{0pt}
     & E3 & $6.213$ & $0.128$ & $36.208$ & $3.017$ & N/A \\\rule{0pt}{2.6ex}\rule[-1.2ex]{0pt}{0pt}
     & E4 & $7.176$ & $0.031$ & $36.208$ & $3.017$ & N/A \\\rule{0pt}{2.6ex}\rule[-1.2ex]{0pt}{0pt}
    19 & E1 & $10.437$ & $0.302$ & $6.880$ & $0.302$ & $0.302$ \\\rule{0pt}{2.6ex}\rule[-1.2ex]{0pt}{0pt}
     & E2 & $13.096$ & $0.115$ & $6.880$ & $0.302$ & $0.302$ \\\rule{0pt}{2.6ex}\rule[-1.2ex]{0pt}{0pt}
     & E3 & $10.570$ & $0.022$ & $6.880$ & $0.302$ & N/A \\\rule{0pt}{2.6ex}\rule[-1.2ex]{0pt}{0pt}
     & E4 & $13.155$ & $0.008$ & $6.880$ & $0.302$ & N/A \\
\hline
\end{tabular}
\end{center}
\vspace{-0.5cm}
\caption{Expectation and variance for $\Y{M}{k}$ for Examples E1, E2, E3, and E4. For comparison, we show the numerical values of the variance bounds from \Cref{th:ldi} (see \Cref{eq:ldi_varbound}) and \Cref{th:support} (see \Cref{eq:support_varbound}). For the examples with diagonal matrices (E1 and E2), we also show the numerical value of the variance bound from \Cref{th:diag} (see \Cref{eq:diag_varbound}).}
\label{tab:examples}
\end{table}

In \Cref{tab:examples}, we show $\E[\Y{M}{k}]$ and $\var(\Y{M}{k})$ for the distributions in \Cref{fig:pdfs}. We first consider the expectation of $\Y{M}{k}$. For all four examples, we observe that $\E[\Y{M}{k}]$ increases with $k$. For all examined
values of $k$, we see that $\E[\Y{M_{\textrm{E4}}}{k}]>\E[\Y{M_{\textrm{E2}}}{k}]>\E[\Y{M_{\textrm{E3}}}{k}]>\E[\Y{M_{\textrm{E1}}}{k}]$. Our observations thus suggest that the expectation of $\Y{M}{k}$ is large when we choose eigenvalues of $M$ uniformly at random from the interval $[1,\kappa]$ and small when we set half of the eigenvalues of $M$ to $1$ and the other half to $\kappa$. 

We now give several observations about the variance of $\Y{M}{k}$. For all examined values of $k$, we see that $\var(\Y{M_{\textrm{E1}}}{k})>\var(\Y{M_{\textrm{E2}}}{k})>\var(\Y{M_{\textrm{E3}}}{k})>\var(\Y{M_{\textrm{E4}}}{k})$. Our observation of larger $\var(\Y{M}{k})$ for the examples with diagonal matrices (Examples E1 and E2) than for the examples with non-diagonal matrices (Examples E3 and E4) gives intuitive support for \Cref{conj:diag}. Our observation that $\var(\Y{M_{\textrm{E1}}}{k})>\var(\Y{M_{\textrm{E2}}}{k})$ reflects the fact that Example E1 maximizes the variance in this case (see \Cref{th:diag}). 

For all examined $k$, the value of the variance bound in \Cref{th:ldi} (see \Cref{eq:ldi_varbound}) is at least 12 times larger than the value of the variance bound in \Cref{th:support} (see \Cref{eq:support_varbound}). For $k=1$ and $k=19$, the cases in which $\wedge_{n,k}=1$, the value of the variance bound in \Cref{th:support} is equal to the value of the variance bound for diagonal positive-definite matrices (\Cref{th:diag}). Additionally, it is sharp in Example E1. 

\begin{figure}[t]
\centering
\includegraphics[trim={0cm 0.25cm 0cm 0.5cm},clip,width=1\textwidth]{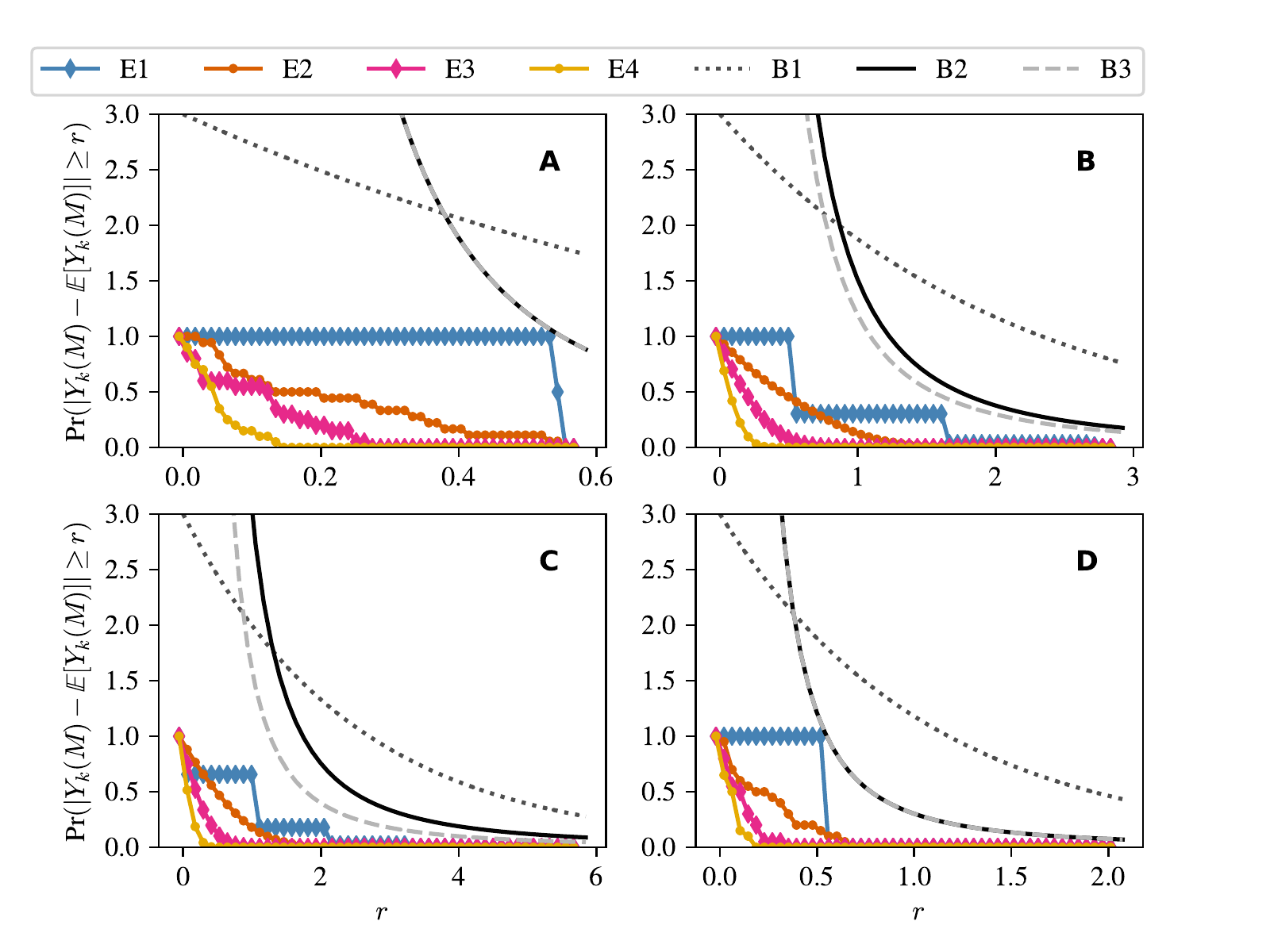}
\caption{Empirical tails of the distribution of log-minors. We show $\prob(|\Y{M}{k}-\mathbbm E[\Y{M}{k}]|\geq r)$ for positive-definite $n\times n$ matrices from examples E1, E2, E3, and E4 with $n=20$ and condition number $\kappa=3$. We show results for $k\times k$ submatrices with \panel{A} $k=1$, \panel{B} $k=5$, \panel{C} $k=10$, and \panel{D} $k=19$. The curve B1 represents the tail bound from \Cref{th:ldi}. The curves B2 and B3 visualize Chebyshev bounds that we obtain from the variance bounds in \Cref{th:support,th:diag}, respectively.}
\label{fig:large_deviations}
\end{figure}

In Fig.\,\ref{fig:large_deviations}, we show the empirical tails $\prob(|\Y{M}{k}-\mathbbm E[\Y{M}{k}]|\geq r)$ for our four examples. We also show the tail bound B1 from \Cref{th:ldi} and two Chebyshev bounds\footnote{We can obtain a tail bound from a variance bound by using Chebyshev's inequality \cite{Blitzstein2014} (page 429), $\operatorname{Pr}(|X-\mathbbm E[X]|\geq r)\leq \frac{\var(X)}{r^2}$, for an integrable random variable $X$ and $r\in\mathbbm R_+$.}, B2 and B3, which we obtain from the variance bounds in \Cref{th:support,th:diag}, respectively. Consistent with our observations in \Cref{tab:examples} on $\var(\Y{M}{k})$, we observe that the tail probability tends to be larger for the examples with diagonal matrices (Examples E1 and E2) than for the examples with non-diagonal matrices (Examples E3 and E4).  

The difference in functional form guarantees that the bound B1 intersects with the Chebyshev bound B2 at two values of $r$. If we denote these values by $r'$ and $r''>r'$, the bound B1 is sharper than B2 on $[0,r']$ and $[r'',\infty)$. In our observations, both bounds exceed the trivial bound $\prob(|\Y{M}{k}-\mathbbm E[\Y{M}{k}]|\geq r)\leq 1$ on $[0,r']$. The value $r''$ lies outside the support of $\prob(|\Y{M}{k}-\mathbbm E[\Y{M}{k}]|\geq r)$. We thus see that B1 is sharper than B2 only for values of $r$ for which neither bound is informative.

For $k=1$ and $k=19$, the bounds B2 and B3 coincide and are sharp at $r=(\log\kappa)/2$ when $M=M_{\textrm{E1}}$. For $k=5$ and $k=10$, the bound (B3) for diagonal positive-definite matrices is sharper than the bound (B2) for general positive-definite matrices. The difference between the two bounds is most visible for $k=10$, which is the case that maximizes $\wedge_{n,k}$.


\section{Estimating mean subsystem entropy}\label{sec:sampling}

We now consider the implications of our results in Section \ref{sec:results} for the problem of estimating the mean subsystem entropy of a given system of coupled variables. When the joint distribution of variables is a multivariate normal distribution, one can compute the differential entropy of a subsystem by applying \Cref{eq:diff_ent} to the corresponding sub-covariance matrix. We are interested in the \emph{mean subsystem entropy} $\E[h(A_k(M))]$ for subsystems of $k$ variables. As we noted previously, the large number of subsystems for even modest values of $n$ and $k$ render it prohibitive to exactly compute $\E[h(A_k(M))]$. Fortunately, the tail and variance bounds in Section \ref{sec:results} allow us to instead provide sampling guarantees, through which one can achieve a prescribed sampling accuracy. We give upper bounds on  the standard error and on the coefficient of variation for both a sample mean of $\Y{M}{k}$ and a sample mean of subsystem entropy.

Fix a subsystem size $k$  and sample size $q\geq 1$. The $q$-sample mean of $\Y{M}{k}$ is
\begin{align*}
    S_Y := \frac{1}{q}\sum_{i = 1}^q \ld{A_i}\,,
\end{align*}
where we choose each $A_i$ uniformly at random from $\mathcal A_k(M)$. The $q$-sample mean of subsystem entropy is
\begin{align*}
    S_h := \frac{1}{q}\sum_{i = 1}^q h(A_i)\,.
\end{align*}
We use $S_Y$ and $S_h$ as estimators of the population means $\E[\Y{M}{k}]$ and $\E[h(A_k(M))]$, respectively. These estimators are unbiased, as $\E[S_Y]=\E[\Y{M}{k}]$ and $\E[S_h]=\E[h(A_k(M))]$. A measure of reliability of an estimator is the \emph{standard error}, which one computes as the estimator's standard deviation. Because $h(A_k(M))$ differs from $\Y{M}{k}/2$ by a constant, the sample mean $S_h$ has the standard error
\begin{align*}
    \hat{\sigma}(S_h) = \frac{1}{2}\hat{\sigma}(S_Y) = \frac{1}{2}\sqrt{\frac{\var \left(\Y{M}{k}\right)}{q}}\,.
\end{align*}
We may therefore use the bounds of \Cref{th:ldi,th:support,th:diag} to derive bounds on the standard error for $S_Y$ and $S_h$. 

\begin{cor}[Standard error of the sample mean subsystem entropy]
Let $M$ be a covariance matrix of an $n$-variate normal distribution, and suppose that the condition number of $M$ satisfies $\kappa(M)\leq\K$. Let $S_h$ be the $q$-sample mean of the entropy of subsets of $k$ variables; and let $S_Y$ be the $q$-sample mean of log-determinants of $k\times k$ principal submatrices of $M$. It then follows, for any subsystem size $k$, that the standard error of the mean subsystem entropy is $\hat\sigma(S_h)=\hat\sigma(S_Y)/2$ and that $\hat\sigma(S_Y)$ satisfies
\begin{equation}
    \hat{\sigma}(S_Y)\leq\sqrt{\frac{6k(n-k)}{qn}}\log \K\,,\label{eq:se1}
\end{equation}
and 
\begin{equation}
    \hat{\sigma}(S_Y)\leq\frac{1}{2}\sqrt{\frac{\wedge_{n,k}}{q}}\log \K\,.\label{eq:se2}
\end{equation}
Furthermore, if $M$ is diagonal,
\begin{equation}
    \hat{\sigma}(S_Y)\leq\frac{1}{2}\sqrt{\frac{k(n-k)}{q(n-1)}}\log\K\,.\label{eq:se3}
\end{equation}
\end{cor}

The \emph{coefficient of variation} $\cv{S}$ is another measure of reliability for estimators. It measures the size of the typical error of an estimator $S$ as a fraction of the magnitude of $\E[S]$. As a formula, it is given by
\begin{align}
    \cv{S} := \frac{\hat{\sigma}(S)}{\E[S]}\,. \label{eq:cv_def}
\end{align}
The coefficient of variation for $S_Y$ arises from the standard deviation of the relative error
\begin{align*}
    \mathcal{E} := \abs{\frac{\Y{M}{k} - \E[\Y{M}{k}]}{\E[\Y{M}{k}]}}\,
\end{align*}
of $\Y{M}{k}$ because
\begin{align*}
    \frac{\hat{\sigma}(S_Y)}{\E[S_Y]}
    =\frac{\hat{\sigma}(\Y{M}{k})- \E[\Y{M}{k}]}{\E[\Y{M}{k}]\sqrt{q}}
    =\hat{\sigma}\left(\left|\frac{\Y{M}{k}-\E[\Y{M}{k}]}{\E[\Y{M}{k}]}\right|\right)/\sqrt{q}
    =\hat{\sigma}(\mathcal E)/\sqrt{q}\,. 
\end{align*}

For a multivariate Gaussian distribution, the following corollaries give bounds on the coefficient of variation for the sample mean of log-minors and for the sample mean of subsystem entropy.

\begin{cor}[Coefficient of variation for a sample mean of log-minors] \label{cor:rel_err}
    Let $\lambda_1(M),\ldots,\lambda_n(M)$ be the eigenvalues of $M$; we order them from largest to smallest. 
    Let $\ell(M) = \min\{\abs{\log \lambda_1(M)}, \abs{\log \lambda_n(M)}\}$. 
    If $\ell(M)\neq0$, the coefficient of variation for a $q$-sample mean $S_Y$ of $\Y{M}{k}$ satisfies 
    \begin{align}
        \cv{S_Y} \leq \frac{\log \hat{\kappa}}{\ell(M)}\sqrt{\frac{6(n-k)}{qkn}} \label{eq:cvy1}
    \end{align}
and
    \begin{align}
        \cv{S_Y} \leq \frac{\log \hat{\kappa}}{2\ell(M)}\sqrt{\frac{\wedge_{n,k}}{qk^2}}\,.\label{eq:cvy2}
    \end{align}
\end{cor}

\begin{proof}
    This corollary follows from \Cref{eq:cv_def}. We use \Cref{eq:ldi_varbound,eq:support_varbound} as upper bounds on the numerator. For all $k$, a lower bound on the denominator is $\abs{\E[Y_k(M)]} \geq k\ell(M)$. 
\end{proof}

\begin{cor}[Coefficient of variation for mean subsystem entropy]
    For an $n$-variate Gaussian distribution with covariance matrix $M$, the coefficient of variation $c_v(S_h)$ for a $q$-sample mean $S_h$ of subsystem entropy satisfies
    \begin{align}
        c_v(S_h) \leq \frac{2\log \hat{\kappa}}{\ell(M) + \log 2e\pi}\sqrt{\frac{6(n-k)}{qkn}} \label{eq:cvh1}
    \end{align}
and
    \begin{align}
        c_v(S_h) \leq \frac{\log \hat{\kappa}}{\ell(M) + \log 2e\pi}\sqrt{\frac{\wedge_{n,k}}{qk^2}}\,.\label{eq:cvh2}
    \end{align}
\end{cor}

\begin{proof}
    We derive this result from \Cref{eq:cv_def}; we use \Cref{eq:ldi_varbound,eq:support_varbound} to bound the numerator, and we use \Cref{eq:diff_ent} to bound the expectation in the denominator. 
\end{proof}

The bounds in \Cref{eq:cvy2,eq:cvh2} are sharper bounds than \Cref{eq:cvy1,eq:cvh1}. From \Cref{eq:cvy2,eq:cvh2}, we see that both $\cv{S_Y}$ and $\cv{S_h}$ decay in proportion to $\sqrt{k}$. Indeed, under a certain regularity condition (which we specify in \Cref{cor:rel_err_concentration}), the coefficient of variation decays to $0$ in the limit of large $n$ and large $k$. 

\begin{cor}[Concentration of the relative error] \label{cor:rel_err_concentration}
    Let $\{M_i\}$ be a sequence of positive-definite matrices of dimension $n(i)$. Let $k = k(i) \leq n(i)$ be a function of $i$. Suppose that the sequence 
    \begin{align}
        a_i := \sqrt{k(i)} \ell(M_i) \label{eq:criterion}
    \end{align}   
    is nondecreasing and unbounded. It then follows that $c_v(S_Y) \rightarrow 0$ and $\mathcal{E}$ converges in probability to $0$ as $i$ becomes large. 
\end{cor}

\begin{remark}
A sufficient condition for the concentration of $\mathcal E$ is that the sequence $\{M_i\}$ has fixed condition number and the smallest eigenvalue $\lambda_n$ is bounded away from both $0$ and $1$. Formally, the latter condition is 
\begin{align}\label{eq:ccon}
    \text{there exists} \,\, \delta > 0 \,\, \text{such that, for all} \,\, i\,, \text{ we have}\,\, \lambda_{n(i)}(M_i) \in [\delta, 1 - \delta]\cap[1+\delta, \infty)\ \,. 
\end{align}
\end{remark}

\begin{remark}
A popular model for sample covariance matrices is the Wishart ensemble\footnote{The Wishart ensemble $W_n(V,n_f)$ with scale matrix $V$ and $n_f$ degrees of freedom is the ensemble of random matrices $M:=n_f^{-1}\sum_{i=1}^{n_f}X_{i}^TX_{i}$, where the $X_1,X_2,\dots,X_{n_f}$ are $n_f$ realizations of an $n$-variate random variable with 0-mean Gaussian distribution $N_n(0,V)$ \cite{Goodman1963,Katzav2010}.}. A sequence $\{M_i\}$ of Wishart matrices can satisfy the condition in \Cref{eq:ccon} if the ratio $c:=n/n_f$ of the number $n$ of variables and the number $n_f$ of degrees of freedom is $c\notin\{1/4,1\}$ \cite{Katzav2010,Majumdar2010}.
\end{remark}

\begin{figure}[tp]
\centering
\includegraphics[trim={0cm 0.75cm 0cm 0.5cm},clip,width=1\textwidth]{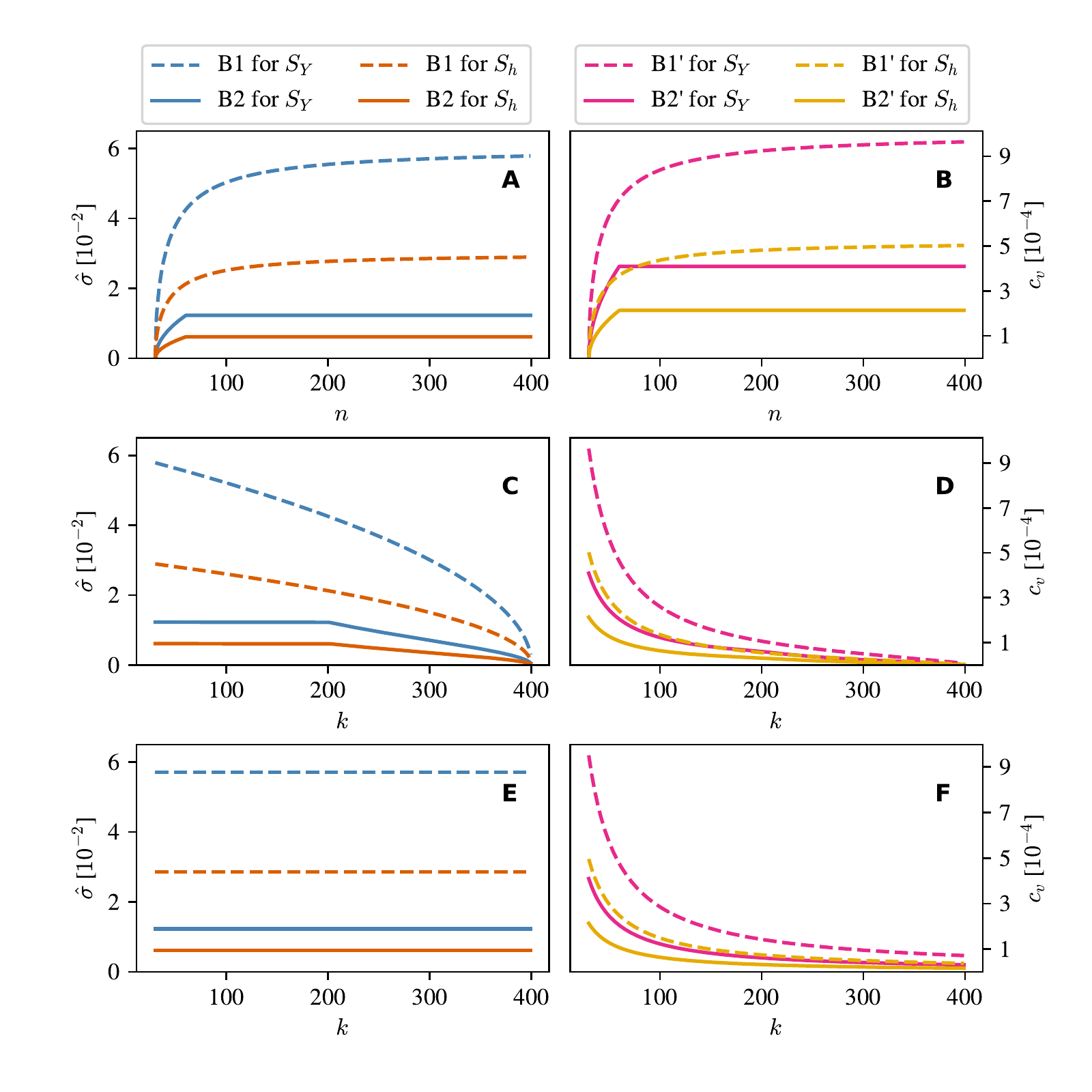}
\caption{Bounds on the standard error and coefficient of variation for a sample mean $S_Y$ of log-minors and a sample mean $S_h$ of subsystem entropy. In the left panels, we show the bounds B1 from \Cref{eq:se1} and B2 from \Cref{eq:se2} on the standard error of $S_Y$ and $S_h$. In the right panels, we show the bounds B1' (see \Cref{eq:cvy1,eq:cvh1}) and B2' (see \Cref{eq:cvy2,eq:cvh2}) on the coefficient of variation of $S_Y$ and $S_h$. We show results for progressively larger $n$ with fixed $k=30$ in panels \panel{A} and \panel{B}; results for progressively larger $k$ with fixed $n=400$ 
in panels \panel{C} and \panel{D}; and results for progressively larger $k$ with a fixed ratio $k/n=0.1$ in panels \panel{E} and \panel{F}. We compute bounds using $q=2000k$, $\K=3$, and $\ell(M)=1$.} 
\label{fig:sampling_error}
\end{figure}

One can use these bounds on the standard error to choose a sample size $q$ that guarantees a desired accuracy of a sample mean. In Fig.\,\ref{fig:sampling_error}, we show our bounds on the standard error and the coefficient of variation of $S_Y$ and $S_h$ with $q=2000k$ and $\ell(M)=1$. In the left panels, we show the bounds B1 from \Cref{eq:se1} and B2 from \Cref{eq:se2} on the standard error of $S_Y$ and $S_h$. In the right panels, we show the bounds B1' (see \Cref{eq:cvy1,eq:cvh1}) and B2' (see \Cref{eq:cvy2,eq:cvh2}) on the coefficient of variation of $S_Y$ and $S_h$. In panels \panel{A} and \panel{B}, we vary the system size $n$ for fixed subsystem size $k$. We observe for $n\leq 2k$ that the values of the bounds B2 and B2' increase with $n$. For $n>2k$, the bounds B2 and B2' are independent of $n$. The bounds B1 and B1' are less sharp than the bounds B2 and B2'. The values of B1 and B1' increase with $n$ and approach their asymptotic values from below. For example, the bound B1 for $\hat\sigma(S_Y)$ has a limiting value of $\sqrt{6k/q}\times\log\K\approx0.06$. In panels \panel{C} and \panel{D}, we vary $k$ for fixed $n$. We observe for $k\leq n/2$ that the value of the bound B2 on the standard error is independent of $k$. For $k>n/2$, the value of B2 decreases with increasing $k$ and is $0$ for $k=n$. The bound B1 is less sharp than B2. Its value decreases with increasing $k$ for any $k\leq n$. The values of the bounds B1' and B2' on the coefficient of variation decrease with increasing subsystem size and are $0$ for $k=n$. In panels \panel{E} and \panel{F}, we vary $k$ for fixed ratio $k/n$, and we observe that the bounds on the standard error are independent of $k$ if the ratio $k/n$ is constant. The values of the bounds B1' and B2' decrease with increasing $k$. This is consistent with our previous observation that $\cv{S_Y}$ vanishes if the sequence $a_i$ (see \Cref{eq:criterion}) becomes unbounded.

It is important to note that all of our bounds on the standard error and the coefficient of variation of $S_Y$ and $S_h$ are asymptotically constant in $n$. It is thus not necessary to sample proportionally more minors from a larger matrix. Instead, to guarantee a desired accuracy of a sample mean of log-minors or subsystem entropy, one can choose $q$ to be a function of $k$. To ensure that the standard error is constant or decreases with growing $n$ and $k$, it is sufficient to choose $q$ in linear proportion to $k$. When the smallest and largest eigenvalues of a system's correlation matrix are fixed, one can ensure that the coefficient of variation is constant or decreasing with growing $n$ and $k$ by choosing $q$ in linear proportion to $k^{-1}$.


\section{Conclusions}\label{sec:conclusions}

We examined the problem of estimating the mean subsystem entropy of a system of $n$ coupled variables with covariance matrix $M$. When the joint distribution of a system's variables is an $n$-variate Gaussian, $t$, or Cauchy distribution, the mean differential entropy of subsystems is an affine function of the log-minors of the covariance matrix \cite{Ahmed1989,Nadarajah2005}. We derived tail and variance bounds on the distribution of log-minors of fixed size of a positive-definite matrix with bounded condition number. Using our variance bounds, we provided upper bounds on the standard error and on the coefficient of variation of both the sample mean of log-minors and the sample mean of subsystem entropy. Our results indicate that, despite the rapid growth of the number of subsystems with $n$, the accuracy of these sample means is asymptotically independent of a system's size. Instead, it is sufficient to increase the number of samples in linear proportion to the size of subsystems to achieve a desired sampling accuracy. 

Our results are salient to studies that use mean subsystem entropy to examine systems of coupled variables \cite{Koetter2003,Tononi1994,Tononi1999}. Even for a system with as few as $50$ variables, sampling just 0.001\% of its subsystem entropies can require the computation of over a billion log-determinants. Using the largest and smallest eigenvalues of a system's covariance matrix to determine the number of samples that are needed to achieve a prescribed accuracy for a sample mean can thus facilitate a quantitative study of mean subsystem entropy when it would otherwise be impossible.

Throughout our paper, we relied only on knowledge of the largest and smallest eigenvalues of a system's covariance matrix. We expect that it is possible to derive sharper bounds than our current results when one knows the complete spectrum of a system's covariance matrix, likely by relying on Cauchy's interlacing theorem (\Cref{prop:cauchy}) to control the log-minors. 

We presented two bounds on the variance of a log-minor that we choose uniformly at random from the set of log-minors of size $k$ of an $n\times n$ positive-definite matrix. The variance bound in \Cref{th:support} is sharper than the one
in \Cref{th:ldi}, but either bound is sufficient to deduce that the accuracy of a sample mean of subsystem entropy is asymptotically independent of a system's size and that one can achieve a prescribed accuracy by choosing the number of samples in linear proportion to the size of subsystems.

The proof of our first bound (see \Cref{sec:proof2}) relies on the existence of an upper bound for the difference between $\ld{A}$ for two different principal submatrices $A\in\mathcal A_k(M)$ and the invariance of $\ld{A}$ under a basis transformation of $A$. The proof of our second bound (see \Cref{sec:proof1}) relies on the existence of an upper bound and a lower bound for the support of the distribution $\p_{M,k}$. 

Similar bounds and the invariance under basis transformation hold for several other matrix properties, including the largest and smallest eigenvalues. It is thus plausible that one can derive similar results for the standard error and coefficient of variation for many spectral properties of principal submatrices. For example, Chatterjee and Ledoux (2009) proved a large-deviation inequality for the empirical cumulative eigenvalue distribution of principal submatrices of Hermitian matrices \cite{Chatterjee2009}. These and other variance and tail bounds on submatrix properties offer welcoming  possibilities to enhance computational studies that characterize complex systems based on the mean properties of their subsystems. For example, they can provide guarantees for linear sketching techniques, which are relevant for data dimensionality reduction. They can also facilitate the use of methods of spectral graph analysis in the study of subgraphs, graphlets, and motifs in networks. 


\section*{Acknowledgements}

We thank Cl\'ement Canonne, Kameron Decker Harris, Michael Neely, and participants of the \textit{IPAM Quantitative Linear Algebra Tutorials} for helpful discussions. A.C.S. was supported by the Clarendon Fund, e-Therapeutics plc, and funding from the Engineering and Physical Sciences Research Council under grant number EP/L016044/1. P.S.C. was supported by the National Science Foundation under Graduate Research Fellowship Grant 1122374. 



\end{document}